\documentclass[12 pt]{article}%
\usepackage{amsmath, amsfonts, amsthm, color,latexsym}
\usepackage{amsmath, ulem}
\usepackage{amsfonts}
\usepackage{amssymb}
\usepackage{color, soul}
\usepackage{xcolor}
\usepackage{tikz-cd}
\usepackage[all]{xy}
\usepackage{graphicx}%
\providecommand{\U}[1]{\protect\rule{.1in}{.1in}}
\allowdisplaybreaks[4]
\newtheorem{theorem}{Theorem}[section]

\newtheorem{remark}[theorem]{Remark}

\newtheorem{lemma}[theorem]{Lemma}
\newtheorem{final remark}[theorem]{Final Remark}

\allowdisplaybreaks[4]

\newcommand {\R}{\mathbb{R}}

\newcommand {\N} {\mathbb{N}}
\newcommand{\norma}[1]{\| #1 \|}
\newcommand{\conj}[2]{\left \{ {#1} \, : \, {#2} \right \}}

\begin{document}

\title{The weak maximizing property for $(L_p([0,1]),L_q([0,1]))$}
\author{Vinícius C. C. Miranda\thanks{Supported by FAPESP (Grants 2025/08630-0 and 2023/12916-1) and FAPEMIG (Grant APQ-01853-23)\newline 2020 Mathematics Subject Classification: 46B20, 46B25.\newline Keywords: weak maximizing property, norm-attaining operator, compact perturbation
property.}}
\date{}
\maketitle

\begin{abstract}
    We solve \cite[Question 4.2]{dantasjung} left open by Dantas, Jung and Martínez-Cervantes by providing a complete characterization for the pairs $(L_p([0,1]),L_q([0,1]))$ having the weak maximizing property. More precisely, we prove that, for $1<p<\infty$ and $1\leq q<\infty$, the pair $(L_p([0,1]),L_q([0,1]))$ has the weak maximizing property if and only if $p=q=2$. A key step in the proof is a Gaussian construction that yields the failure of the compact perturbation property for the pair $(\ell_2,L_r([0,1]))$ whenever $r>2$.
\end{abstract}

\section{Introduction}

In \cite{danieleduardo}, Pellegrino and Teixeira proved that, for each $1 < p < \infty$ and each $1 \leq q < \infty$, a bounded linear $T: \ell_p \to \ell_q$ attains its norm if and only if there exists a non-weakly null maximizing sequence $(x_n)_n \subset S_{\ell_p}$. Recall that a maximizing sequence for a bounded linear operator $T$ is a sequence $(x_n)_n \subset S_X$ such that $\displaystyle \lim_{n \to \infty} \norma{T(x_n)} = \norma{T}$.  This result motivated the introduction of the weak maximizing property, defined by Aron, García, Pellegrino and Teixeira in \cite{arongarciapelteix}: a pair of Banach spaces $(X, Y)$ satisfies the {\bf weak maximizing property} ($WMP$, in short) if every bounded linear operator $T: X \to Y$ admitting  a non-weakly null maximizing sequence attains its norm. Since then, this property has also been studied in \cite{dantasjung, garcia-lirola, jung}.

In \cite{dantasjung}, Dantas, Jung and Martínez-Cervantes asked whether the pair $(L_p([0,1]), L_q([0,1]))$ has the $WMP$ whenever $1 < p \leq 2 \leq q < \infty$ with $p \neq q$. The main purpose of this manuscript is to provide a negative answer to this question. 

\medskip

\noindent \textbf{Main Result} Let $1 < p < \infty$ and $1 \leq q < \infty$.
The pair of Banach spaces $(L_p([0,1]), L_q([0,1]))$ has the $WMP$ if and only if $p=q=2$.

\medskip

The main idea behind the proof is to construct a suitable compact perturbation of the standard Gaussian embedding of $\ell_2$ into $L_q([0,1])$, obtained from a sequence of independent standard Gaussian random variables (see, e.g., \cite[Proposition 6.4.12]{albiackalton}). More precisely, we replace the first Gaussian coordinate by a constant function and slightly modify the remaining coordinates in order to obtain a non-norm-attaining operator whose norm is strictly larger than that of its noncompact part. Accordingly, in Section 2 we collect the probabilistic preliminaries and prove the technical lemmas required for this construction. In Section 3, we carry out the construction and use it to prove the Main Result.

%%We also note that variants of the property were studied by Chakraborty \cite{cha} concerning the minimum norm, and by Luiz and the second named author in the setting of Banach lattices \cite{luizmiranda}.

\section{Background in Probability Theory}

In this section, we recall some basic facts from probability theory. Although some of the results below are standard, we include their statements and, when appropriate, their proofs in order to make the paper self-contained and accessible to readers primarily interested in functional analysis. We refer the reader to \cite{rolla,Kallenberg2002, Billingsley1995, talagrand} for further details.

\begin{remark} \label{remark1} \rm Let $(\Omega, \Sigma, \mathbb{P})$ be a probability space and let $X : \Omega \to \R$ be a random variable, i.e., a measurable function. \\
(i) The probability space induced by $X$ is the measure space $(\R, {\cal B}_{\R}, \mathbb{P}_{X})$, where ${\cal B}_{\R}$ is the Borel $\sigma$-algebra of $\R$ and 
$$ \mathbb{P}_{X}(B) := \mathbb{P}(\conj{\omega \in \Omega}{X(\omega) \in B}), \qquad B \in {\cal B}_{\R}. $$
\noindent (ii) The comulative distribution function of $X$ is the function $F_X: \R \to [0,1]$ defined by 
$$ F_X(x) = \mathbb{P}(\conj{\omega\in \Omega}{X(\omega) \leq x}). $$
If two random variables $X$ and $Y$ satisfy $F_X=F_Y$, then $\mathbb{P}_{X}=\mathbb{P}_{Y}$ (see \cite[Theorem 3.10]{rolla}). In this case, we write $X\overset{d}{=}Y$. For instance, we say that $X$ has the Gaussian distribution, and write $X \sim \mathcal{N}(\mu,\sigma^2)$, if 
$F_X (x) = \displaystyle  \dfrac{1}{\sqrt{2\pi \sigma^2}} \int_{-\infty}^x e^{\frac{-(t - \mu)^2}{2\sigma^2}} \, dt$.
Whenever $\mu=0$ and $\sigma^2 =1$, we say that $X$ has the standard Gaussian distribution.
\\
\noindent (iii) The expectation of $X\in L_1(\Omega)$ is defined by
$$ \mathbb{E}[X] = \int_{\Omega}X(\omega)\,d\mathbb{P}(\omega).$$
It is easy to check that if $X\overset{d}{=}Y$, then
$\mathbb{E}[\varphi(X)]=\mathbb{E}[\varphi(Y)]$
for every Borel measurable function $\varphi:\R\to\R$ such that $\varphi(X)$ and $\varphi(Y)$ are integrable. In particular, if $X\sim\mathcal{N}(0,1)$, then
$\mathbb{E}[X]=0$ and $\mathbb{E}[X^2]=1$
(see \cite[Example 5.41]{rolla}). \\
\noindent (iv) The variance of $X \in L_2(\Omega)$ is
$ \operatorname{Var} (X) = \mathbb{E}[X^2] - (\mathbb{E}[X])^2. $ The standard deviation of $X$ is 
$\sqrt{\operatorname{Var}(X)}. $
In particular, if $X\sim\mathcal{N}(\mu,\sigma^2)$, then $\operatorname{Var}(X)=\sigma^2$.\\
(v)  The random variables $X_1,\ldots,X_n$ are said to be independent if
$$ \mathbb{P}(X_1 \in B_1, \dots, X_n \in B_n) = \mathbb{P}(X_1 \in B_1) \cdots \mathbb{P}(X_n \in B_n)$$
for all $B_1,\ldots,B_n\in{\cal B}_{\R}$.
In particular, if
$X_1,\ldots,X_n$ are integrable, then
$\mathbb{E}[X_1\cdots X_n] =
\mathbb{E}[X_1]\cdots\mathbb{E}[X_n].$
Moreover, a sequence $(X_n)_n$ of random variables is said to be independent if every finite subfamily of $(X_n)_n$ is independent.\\
(vi) For every Borel measurable function $h:\R\to\R$ such that $h(G)$ is integrable, it follows that
$$ \mathbb{E}[h(G)] = \frac{1}{\sqrt{2\pi}}
\int_{\R}h(x)e^{-x^2/2}\,dx,$$
whenever $G\sim\mathcal{N}(0,1)$.\\
(vii) Let $f:[a,b]\times\R\to\R$ be a continuous function such that
$\frac{\partial f}{\partial t}$ exists and is continuous, and let
$X:\Omega\to\R$ be a random variable. Suppose that $f(t,X)$ is integrable for every $t\in[a,b]$ and that there exists $Y\in L_1(\Omega)$ such that
$$
\left|
\frac{\partial f}{\partial t}(t,X)
\right|
\leq Y
\quad\text{almost surely, for every $t\in[a,b]$}.
$$
Then the function
$t\longmapsto\mathbb{E}[f(t,X)]$
is differentiable on $(a,b)$ and
$$
\frac{d}{dt}\mathbb{E}[f(t,X)]
=
\mathbb{E}
\left[
\frac{\partial f}{\partial t}(t,X)
\right], \qquad \text{ for every } t\in(a,b).$$ At $t=a$ and $t=b$, the corresponding one-sided derivatives exist and satisfy the same formula
(see \cite[Theorem 5.49]{rolla}). \\
{\rm (viii)} Given $r> 2$ and $G \sim \mathcal{N}(0,1)$, we have
\begin{align*}
    \mathbb{E}[|G|^r] & = \frac{1}{\sqrt{2\pi}} \int_\R |x|^r e^{-x^2/2} \, dx = \frac{2}{\sqrt{2\pi}} \int_0^\infty x^r e^{-x^2/2} \, dx \\
    & = \frac{2}{\sqrt{2\pi}} \int_0^\infty x^{r-1} xe^{-x^2/2} \, dx \\
    & = \frac{- 2}{\sqrt{2\pi}} \int_0^\infty x^{r-1} \left ( e^{-x^2/2} \right )' \, dx \\ 
    & = \frac{- 2}{\sqrt{2\pi}} \left [ \lim_{N \to \infty} N^{r-1} e^{-N^2/2} - 0\right ] + \frac{2}{\sqrt{2\pi}} \int_0^\infty (r-1)x^{r-2} e^{-x^2/2} \, dx \\
    & = 0 + (r-1)\frac{2}{\sqrt{2\pi}} \int_0^\infty x^{r-2} e^{-x^2/2} \, dx  =  (r-1) \mathbb{E} [|G|^{r-2}].
\end{align*}
\end{remark}

Throughout the remainder of this paper, we consider the probability space $([0,1],{\cal B}_{[0,1]},m)$, where $m$ denotes the Lebesgue measure. We shall use a sequence $(g_n)_n$ of independent standard Gaussian random variables on $[0,1]$. For the existence of such a sequence, see, for instance, \cite[Theorem 3.19]{Kallenberg2002}. We also recall that, by \cite[Proposition 6.4.12]{albiackalton}, for every $1\leq p<\infty$, the closed linear span of $\{g_n:n\in\N\}$ in $L_p([0,1])$ is isometrically isomorphic to $\ell_2$ via the mapping
$$ (a_n)_n \in \ell_2 \mapsto \norma{g_1}_p^{-1}  \sum_{n=1}^\infty a_n g_n \in L_p([0,1]). $$
Notice also that 
$\norma{g_n}_p^p = \mathbb{E}[|g_n|^p] = \mathbb{E}[|G|^p]$ holds
for every $n \in \N$, where $G \sim \mathcal{N}(0,1)$. We next prove some technical lemmas that will be used in the proof of our main result.

\begin{lemma} \label{lema1}
    Let $(g_n)_{n}$ be a sequence of independent standard Gaussian random variables on $[0,1]$, $(b_n)_n \in \ell_2$, $r > 0$, and $c \in \R$. Then $c \boldsymbol{1} + \displaystyle \sum_{n=1}^\infty b_n g_n \overset{d}{=} c \boldsymbol{1} + \left (\displaystyle \sum_{n=1}^\infty b_n^2\right )^{1/2} G$, where $G \sim \mathcal{N}(0,1)$. Moreover, 
    $$ \mathbb{E} \left [ \left | c \boldsymbol{1} + \displaystyle \sum_{n=1}^\infty b_n g_n \right |^r \right ] = \mathbb{E} \left [ \left | c \boldsymbol{1} + \left (\displaystyle \sum_{n=1}^\infty b_n^2\right )^{1/2} G \right |^r \right ] = \mathbb{E} \left [ \left | |c| \boldsymbol{1} + \left (\displaystyle \sum_{n=1}^\infty b_n^2\right )^{1/2} G \right |^r \right ]. $$
\end{lemma}

\begin{proof} 
    Setting $X_N := \sum_{n=1}^N b_n g_n$ for each $N \in \N$, we have
$\displaystyle \operatorname{Var}(X_N) = \mathbb{E}[X_N^2] = \sum_{n=1}^N b_n^2 $
since $(g_n)_{n}$ is a sequence of independent standard Gaussian random variables (see, e.g., \cite[p. 238]{diestel}). As $X_N$ converges 
in $L_2([0,1])$ to $\displaystyle \sum_{n=1}^\infty b_n g_n$, we obtain 
\begin{align*}
      \mathbb{E}\left [ \left ( \sum_{n=1}^\infty b_n g_n \right )^2  \right ] & = \left \| \sum_{n=1}^\infty b_n g_n \right \|_2^2 = \lim_{N \to \infty} \left \| X_N \right \|_2^2 = \lim_{N \to \infty} \mathbb{E}[X_N^2] = \sum_{n =1}^\infty b_n^2.
\end{align*}
Let $\sigma := \left (\displaystyle \sum_{n =1}^\infty b_n^2\right )^{1/2} $. If $\sigma = 0$, then $b_n = 0$ for every $n \in \N$, and the result is trivial. If $\sigma > 0$, we have $\displaystyle \sum_{n =1}^\infty b_n g_n \sim \mathcal{N}(0, \sigma^2)$, and so 
$\displaystyle  c \boldsymbol{1} + \sum_{n =1}^\infty b_n g_n \sim \mathcal{N}(c, \sigma^2). $
Therefore, $$\displaystyle c \boldsymbol{1} + \sum_{n =1}^\infty b_n g_n  \overset{d}{=} c\boldsymbol{1} + \left (\sum_{n =1}^\infty b_n^2 \right )^{1/2} G, $$ where $G \sim \mathcal{N}(0,1)$.

Finally, since $\varphi(t) = |t|^r$ is a Borel measurable function, the equality in distribution implies that
$$ \mathbb{E} \left [ \left | c \boldsymbol{1} + \displaystyle \sum_{n=1}^\infty b_n g_n \right |^r \right ] = \mathbb{E} \left [ \left | c \boldsymbol{1} + \left (\displaystyle \sum_{n=1}^\infty b_n^2\right )^{1/2} G \right |^r \right ]. $$
We prove the last identity by considering two cases.  First, notice that if $c \geq 0$, then $ |c \boldsymbol{1} + \sigma G|^r = ||c|\boldsymbol{1} + \sigma G|^r$, which implies that
$ \mathbb{E}[|c \boldsymbol{1} + \sigma G|^r] = \mathbb{E} [||c|\boldsymbol{1} + \sigma G|^r]. $
If $ c < 0$, then 
$$ |c\boldsymbol{1} + \sigma G|^r = |- |c| \boldsymbol{1} + \sigma G|^r = ||c| - \sigma G|^r. $$
So, recalling that the standard Gaussian distribution is symmetric, that is, $G \overset{d}{=}-G$, we have $|c| \boldsymbol{1} + \sigma G \overset{d}{=} |c| \boldsymbol{1} - \sigma G$, and hence
$ \mathbb{E}[|c\boldsymbol{1} + \sigma G|^r] = \mathbb{E} [||c| - \sigma G|^r] = \mathbb{E}[||c| \boldsymbol{1} + \sigma G|^r].$
\end{proof}

\begin{lemma} \label{lema2prob}
    Let $r>2$, let $G\sim \mathcal{N}(0,1)$, and let $c\in\mathbb{R}$. Then the function
    $$ F(t) := \mathbb{E} \left [ |c+tG|^r \right ], \qquad t \in [0,\infty), $$
    is nondecreasing on $[0,\infty)$ and strictly increasing on $(0,\infty)$.
\end{lemma}

\begin{proof}
We first observe that $F$ is continuous on $[0,\infty)$. To see this, let $t_k \to t$ in $[0,\infty)$. Then $ |c+t_k G|^r \to |c+t G|^r $
almost surely. On the other hand, since $(t_k)_k$ is bounded, there exists $M > 0$
such that $|t_k| \leq M$ for every $k \in \N$, and hence
$$        |c+t_kG|^r       \leq        (|c|+M|G|)^r        \leq        2^{r-1}\left(|c|^r+M^r|G|^r\right).$$
Now, since  $\mathbb{E}[|G|^r]<\infty$ (see \cite[Proposition 6.4.11]{albiackalton}), the Dominated Convergence Theorem gives $F(t_k) \to F(t)$, proving that $F$ is continuous on $[0,\infty)$.

We now study the sign of $F'$. For each $t>0$, since $G\sim\mathcal{N}(0,1)$, it holds that
$ \displaystyle F(t) = \frac{1}{\sqrt{2\pi}} \int_{\mathbb{R}} |c+tx|^r e^{-x^2/2}\,dx.
$
By Remark \ref{remark1}(vii), we may differentiate under the expectation and obtain
$$ F'(t) = r\mathbb{E}\left[ |c+tG|^{r-2}(c+tG)G \right]. $$
Setting the auxiliary function $\psi_t(x):=|c+tx|^{r-2}(c+tx),$ we have
$$ F'(t) =  r\mathbb{E}\left[ |c+tG|^{r-2}(c+tG)G \right] = r\mathbb{E}[G\psi_t(G)].$$
Using again that $G \sim \mathcal{N}(0,1)$ and that $\dfrac{d}{dx} (e^{-x^2/2}) = -xe^{-x^2/2}$, we get
$$ F'(t) = r\mathbb{E}[G\psi_t(G)] = \frac{r}{\sqrt{2\pi}} \int_{\mathbb R} x\psi_t(x)e^{-x^2/2}\,dx = - \frac{r}{\sqrt{2\pi}} \int_{\mathbb R} \psi_t(x)\left(e^{-x^2/2}\right)'\,dx. $$
Next, integrating by parts, we obtain
\begin{align*}
    F'(t) & = - \frac{r}{\sqrt{2\pi}} \lim_{N \to \infty}  \left [ \psi_t(x) e^{-x^2/ 2} \big |_{-N}^N  \right ] + \frac{r}{\sqrt{2\pi}}
\int_{\mathbb R}
\psi_t'(x)e^{-x^2/2}\,dx \\
& = - \frac{r}{\sqrt{2\pi}} \lim_{N \to \infty}   \left [ \psi_t(N) e^{-N^2/ 2} - \psi_t(-N) e^{-N^2/ 2}  \right ]  + \frac{r}{\sqrt{2\pi}}
\int_{\mathbb R}
\psi_t'(x)e^{-x^2/2}\,dx \\
& = \frac{r}{\sqrt{2\pi}}
\int_{\mathbb R}
\psi_t'(x)e^{-x^2/2}\,dx = r \mathbb{E}[\psi_t'(G)].
\end{align*}
Finally, since $\psi_t'(x) = t(r-1)|c+tx|^{r-2},$
it follows that
\begin{align*}
    F'(t) & = r t (r-1) \mathbb{E}\left[
|c+tG|^{r-2} \right].
\end{align*}
Since $t>0$, $r > 2$, and $\mathbb{E}\left[
|c+tG|^{r-2} \right]>0,$
we conclude that
$ F'(t)>0$ for every $t>0$. Hence $F$ is strictly increasing on $(0,\infty)$. Since $F$ is continuous on $[0,\infty)$, it follows that $F$ is increasing on $[0,\infty)$.
\end{proof}

\begin{lemma} \label{lema3}
    Let $ r > 2$, let $G \sim \mathcal{N}(0,1)$,  and let $\alpha > 0$ such that $1 < \alpha^2 < r-1$. Define $\phi: [0,1] \to \R$ by 
    $$ \phi(x) = \mathbb{E}[|\alpha \sqrt{x} + \sqrt{1-x} \, G|^r]. $$ Then $\phi$ is continuous on $[0,1]$. Moreover, every point at which $\phi$ attains its maximum belongs to $(0,1)$.
\end{lemma}

\begin{proof} To prove that $\phi$ is continuous, take $x_k\to x$ in $[0,1]$. Then
$$ \alpha\sqrt{x_k}+\sqrt{1-x_k}\,G\longrightarrow\alpha\sqrt{x}+\sqrt{1-x}\,G$$
almost surely. Moreover, since
$\left| \alpha\sqrt{x_k}+\sqrt{1-x_k}\,G \right|^r \leq   2^{r-1}(\alpha^r+|G|^r)$
holds for every $k \in \N$ and $\mathbb{E}[|G|^r] < \infty$, the Dominated Convergence Theorem gives
$$ \lim_{k \to \infty}\phi(x_k) =  \lim_{k \to \infty} \mathbb{E}[\left| \alpha\sqrt{x_k}+\sqrt{1-x_k}\,G \right|^r] = \mathbb{E}[\left| \alpha\sqrt{x}+\sqrt{1-x}\,G \right|^r] = \phi(x). $$
Thus, $\phi$ is continuous on $[0,1]$. 

Now, let $x_0$ be a point such that $\phi(x_0) = \displaystyle \max_{x \in [0,1]} \phi(x)$. To see that $x_0 \neq 0$, we define, for each $t > 0$, $x(t) = \dfrac{t^2}{1+t^2}$. Thus, $x(t) \in (0,1)$, $\displaystyle \lim_{t \to 0^+} x(t) = 0$, $\displaystyle \sqrt{x(t)}=\frac{t}{\sqrt{1+t^2}}$, and $\displaystyle \sqrt{1-x(t)}=\frac{1}{\sqrt{1+t^2}}.$
Hence
\begin{align*}
    \phi(x(t)) = \mathbb{E} \left [  \left | \frac{\alpha \, t}{\sqrt{1+t^2}} + \frac{1}{\sqrt{1+t^2}} G \right |^r\right ] = (1+t^2)^{-r/2}\mathbb{E} \left [ \left | \alpha t + G \right |^r\right ].
\end{align*}
By Remark \ref{remark1}(vii), the function
$H(s)=\mathbb{E}\left[|G+s|^r\right]$
is of class $C^2$. Therefore, Taylor's formula at $s=0$, applied with $s=\alpha t$, gives, as $t\to 0^+$,
$$H(\alpha t)= H(0)+\alpha tH'(0)+\frac{\alpha^2t^2}{2}H''(0)+o(t^2), $$
where $o(t^2)$ denotes a quantity satisfying
$\displaystyle  \lim_{t\to 0^+}\frac{o(t^2)}{t^2}=0. $ 
Since
$H'(0) = r\mathbb{E}\left[|G|^{r-2}G\right]$
and
$H''(0) =
r(r-1)\mathbb{E}\left[|G|^{r-2}\right],$
we obtain
$$ \mathbb{E}\left[|G+\alpha t|^r\right] = \mathbb{E}[|G|^r] +
r\alpha t\mathbb{E}\left[|G|^{r-2}G\right] +
\frac{r(r-1)}{2}\alpha^2t^2
\mathbb{E}\left[|G|^{r-2}\right] +
o(t^2). $$
%%By item (viii), applied on compact intervals, and using the fact that Gaussian random variables have finite moments of every order, we may differentiate under the expectation. Indeed, if $|s|\leq M$, then
%$$ \left| \frac{\partial}{\partial s}|x+s|^r \right| = r|x+s|^{r-1} \leq C_M(1+|x|^{r-1}), $$ and $$ \left| \frac{\partial}{\partial s} \left( r|x+s|^{r-2}(x+s) \right) \right| = r(r-1)|x+s|^{r-2} \leq C_M(1+|x|^{r-2}), $$ for some constant $C_M>0$. Therefore $H$ is twice differentiable and $$ H'(s) = r\mathbb{E}\left[ |G+s|^{r-2}(G+s) \right], $$ while $$ H''(s) = r(r-1)\mathbb{E}\left[ |G+s|^{r-2} \right]. $$
As 
$ \displaystyle \mathbb{E}\left[|G|^{r-2}G\right]= \frac{1}{\sqrt{2\pi}} \int_\R |x|^{r-2} x e^{-x^2/2} \, dx = 0, $
we have
$$ \mathbb{E}\left[|G+\alpha t|^r\right] = \mathbb{E}[|G|^r] + \frac{r(r-1)}{2}\alpha^2t^2 \mathbb{E}\left[|G|^{r-2}\right]
+ o(t^2).
$$
On the other hand, applying Taylor's formula to the function $t\mapsto (1+t)^{-r/2}$ at $t=0$, we get, as $t\to 0^+$,
$$ (1+t^2)^{-r/2} = 1-\frac r2t^2+o(t^2). $$
Combining the two Taylor expansions above, we get
\begin{align*}
\phi(x(t)) & = (1+t^2)^{-r/2}\, 
\mathbb{E}\left[|G+\alpha t|^r\right] \\
& =
\left(1-\frac r2t^2+o(t^2)\right)
\left( \mathbb{E}[|G|^r] +
\frac{r(r-1)}{2}\alpha^2t^2
\mathbb{E}\left[|G|^{r-2}\right] +
o(t^2) \right) \\
& = \mathbb{E}[|G|^r] -\frac{r}{2}t^2 \mathbb{E}[|G|^r] + \frac{r(r-1)}{2} \alpha^2t^2 \mathbb{E}[|G|^{r-2}] + o(t^2) \\
& = \mathbb{E}[|G|^r] -\frac{r}{2}t^2 \mathbb{E}[|G|^r] + \frac{r}{2} \alpha^2 t^2 \mathbb{E}[|G|^r] + o(t^2),
\end{align*}
where in the last equality we used Remark \ref{remark1}(viii). Therefore, 
$$ \phi(x(t)) =
\phi(0) + \frac r2(\alpha^2-1)\mathbb{E}[|G|^r]t^2 + o(t^2). $$
Since $\alpha^2>1$, it follows that
$\phi(x(t))>\phi(0)$
for every $t>0$ sufficiently small. Since $x(t)\in(0,1)$, this shows that $0$ cannot be a maximum point of $\phi$. Therefore $x_0\neq 0$.

To see that $x_0 \neq 1$, consider, for each $t> 0$, $z(t) = \dfrac{1}{1+t^2}$. Then $z(t) \in (0,1)$, $\displaystyle \lim_{t \to 0^+} z(t) = 1$, $\displaystyle \sqrt{z(t)} = \frac{1}{\sqrt{1 + t^2}}$, and $\displaystyle \sqrt{1-z(t)} = \frac{t}{\sqrt{1+t^2}}$. Hence
\begin{align*}
\phi(z(t)) =
\mathbb{E}\left[
\left|
\frac{\alpha}{\sqrt{1+t^2}}
+
\frac{t}{\sqrt{1+t^2}}G
\right|^r
\right] =
(1+t^2)^{-r/2}
\mathbb{E}\left[|\alpha+tG|^r\right].
\end{align*}

We again use Taylor expansions. First, by Remark \ref{remark1}(vii), the function $K(t) = \mathbb{E}\left[|\alpha+tG|^r\right]$
is of class $C^2$ in a neighborhood of $0$ and its Taylor expansion at $t = 0$ is
$$ \mathbb{E}\left[|\alpha+tG|^r\right]
= \alpha^r + \frac{r(r-1)}{2}\alpha^{r-2}t^2 +
o(t^2). $$
 On the other hand,
$$ (1+t^2)^{-r/2} = 1-\frac{r}{2}t^2+o(t^2). $$
Combining the two Taylor expansions above, we obtain
\begin{align*}
\phi(z(t))
&=
\left(
1-\frac{r}{2}t^2+o(t^2)
\right)
\left(
\alpha^r
+
\frac{r(r-1)}{2}\alpha^{r-2}t^2
+
o(t^2)
\right) \\
&=
\alpha^r
-
\frac{r}{2}\alpha^rt^2
+
\frac{r(r-1)}{2}\alpha^{r-2}t^2
+
o(t^2) \\
&=
\phi(1)
+
\frac{r}{2}\alpha^{r-2}
\left(
r-1-\alpha^2
\right)t^2
+
o(t^2).
\end{align*}
Since $\alpha^2<r-1$, it follows that
$\phi(z(t))>\phi(1)$
for all sufficiently small $t>0$. Since $z(t)\in(0,1)$, this shows that $1$ cannot be a maximizer of $\phi$. Therefore, $x_0\neq 1$.
\end{proof}

\section{Proof of the Main Result}

To prove our main result, we first prove that the pair $(\ell_2, L_r([0,1]))$ does not have the $WMP$ whenever $r> 2$. In fact, we prove that this pair fails the compact perturbation property. We recall from \cite{jung} that a pair of Banach spaces $(X,Y)$ is said to have the \textbf{compact perturbation property} ($CPP$, for short) if $T+K$ is norm-attaining whenever $T:X\to Y$ is a bounded linear operator and $K:X\to Y$ is a compact operator such that $\|T\|<\|T+K\|$.

\begin{lemma} \label{lema2}
    The pair of Banach spaces $(\ell_2, L_r([0,1]))$ fails to have the
    $CPP$ for every $r > 2$.
\end{lemma}

\begin{proof} We prove first the real case.
    Fix $r > 2$ and choose $\alpha > 0$ with $1 < \alpha^2 < r - 1$. Let $(g_n)_{n\geq 2}$ be a sequence of independent standard Gaussian random variables on $[0,1]$ and let $(\lambda_n)_{n \geq 2}$ be a sequence of real numbers such that $0 < \lambda_n < 1$ holds for every $n \geq 2$ and $\displaystyle \lim_{n \to \infty} \lambda_n = 1$. By \cite[Theorem 6.4.12]{albiackalton}, the map $\displaystyle T((a_j)_j) = \sum_{n=2}^\infty \lambda _n a_n g_n$ defines a bounded linear operator from $\ell_2$ into $L_r([0,1])$. Define also the rank-one operator $K: \ell_2 \to L_r([0,1])$ by 
    $K((a_j)_j) = \alpha a_1 \boldsymbol{1},$
    where $\boldsymbol{1}$ denotes the constant function equal to $1$ on $[0,1]$.

We claim that $S:= T+K$ does not attain its norm and that $\norma{S} > \norma{T}$.
Take $a = (a_j)_j \in S_{\ell_2}$. By Lemma \ref{lema1}, applied to $c = \alpha a_1$ and $b_n = \lambda_n a_n$, $n \geq 2$, we have
$$ \norma{Sa}_r^r = \mathbb{E} \left [ \left |  \alpha |a_1| \boldsymbol{1}  + \left ( \sum_{n=2}^\infty \lambda_n^2 a_n^2 \right )^{1/2} G \right |^r \right ]$$
where $G\sim \mathcal{N}(0,1)$. By Lemma \ref{lema2prob} and the inequality
$\left(\sum_{n=2}^{\infty}\lambda_n^2a_n^2\right)^{1/2}
    \leq
    \left(\sum_{n=2}^{\infty}a_n^2\right)^{1/2}
    =
    \sqrt{1-a_1^2}$, we get that
\begin{align*}
    \norma{Sa}_r^r \leq \mathbb{E} \left [ \left |\alpha |a_1| +  \left(\sum_{n=2}^{\infty}a_n^2\right)^{1/2} G  \right |^r \right ] = \mathbb{E} \left [ \left | \alpha \sqrt{a_1^2} + \sqrt{1-a_1^2} \, G \right |^r \right ].
\end{align*}
Thus, defining $\phi:[0,1]\to\mathbb{R}$ by
 $\phi(x)    =    \mathbb{E}\left[    \left|    \alpha\sqrt{x}    +    \sqrt{1-x}\,G   \right|^r   \right],$
we have
$$ \norma{Sa}_r^r \leq \phi(a_1^2) \leq \sup_{x \in [0,1]} \phi(x), $$
and consequently $\norma{S}^r \leq \displaystyle \sup_{x \in [0,1]} \phi(x)$. By Lemma \ref{lema3}, $\phi$ is continuous on $[0,1]$, and so there exists $x_0\in[0,1]$ such that $\phi(x_0) = \displaystyle \max_{x\in [0,1]} \phi(x)$. Again by Lemma \ref{lema3}, we have $x_0 \in (0,1)$. Moreover, $\norma{S}^r \leq \phi(x_0)$.
To see the reverse inequality, consider the sequence $(z_m)_{m \geq 2} \subset S_{\ell_2}$ given by
$$ z_m = \sqrt{x_0} \, e_1 + \sqrt{1-x_0} \, e_m. $$
Then 
$$ \norma{Sz_m}_r^r = \norma{ \alpha \sqrt{x_0} \, \boldsymbol{1} + \lambda_m \sqrt{1-x_0} \, g_m }_r^r = \mathbb{E} \left [ | \alpha \sqrt{x_0} + \lambda_m \sqrt{1-x_0} \, G |^r \right ] $$
holds for every $m \geq 2$. Since $\lambda_m \to 1$, the Dominated Convergence Theorem yields
$$ \norma{S}^r \geq  \lim_{m \to \infty} \norma{Sz_m}_r^r = \mathbb{E} \left [ | \alpha \sqrt{x_0} + \sqrt{1-x_0} \, G |^r \right ] = \phi(x_0). $$
Therefore $\norma{S}^r = \phi(x_0)$. We now prove that $S$ is not norm-attaining. Let $a=(a_j)_j\in S_{\ell_2}$ and set $x=a_1^2$. If $x=1$, then $\norma{Sa}_r^r=\phi(1)<\phi(x_0)=\norma{S}^r$. Suppose now that $x<1$. Since $\displaystyle \sum_{n=2}^{\infty}a_n^2=1-x>0$, there exists $n\geq 2$ such that $a_n\neq 0$. Since $0<\lambda_n<1$ for every $n\geq 2$, it follows that
$\displaystyle \sum_{n=2}^{\infty}a_n^2(1-\lambda_n^2)>0,$ so
$\displaystyle \sum_{n=2}^{\infty}\lambda_n^2a_n^2<\sum_{n=2}^{\infty}a_n^2=1-x.$ Moreover, since $\lambda_n>0$ for every $n\geq 2$, we also have
$$0<\left(\sum_{n=2}^{\infty}\lambda_n^2a_n^2\right)^{1/2}<\sqrt{1-x}.$$ Thus, by the strict monotonicity of the map
$t\longmapsto\mathbb{E}\left[\left|\alpha\sqrt{x}+tG\right|^r\right]$
on $(0,\infty)$, given by Lemma \ref{lema2prob}, we obtain
$$ \norma{Sa}_r^r = \mathbb{E} \left [ \left | \alpha \sqrt{x} + \left ( \sum_{n=2}^\infty \lambda_n^2 a_n^2 \right )^{1/2} G  \right |^r \right ] < \phi(x) \leq \phi(x_0) = \norma{S}^r.$$
Therefore $S$ does not attain its norm.

Finally, let us prove that $\norma{S}>\norma{T}$. By \cite[Theorem 6.4.12]{albiackalton}, for every $a=(a_j)_j\in\ell_2$, we have
$$
\norma{Ta}_r
= \norma{\sum_{n=2}^{\infty}\lambda_na_ng_n}_r =  \norma{g_2}_r \left(\sum_{n=2}^{\infty}\lambda_n^2a_n^2\right)^{1/2} \leq \norma{g_2}_r\norma{a}_2. $$
Therefore, $\norma{T}\leq\norma{g_2}_r$.
On the other hand, for every $m\geq 2$, we have
$$\norma{T}\geq\norma{Te_m}_r=\lambda_m\norma{g_m}_r=\lambda_m\norma{g_2}_r.$$
Since $\lambda_m\to 1$, $\norma{T}\geq\norma{g_2}_r$, and hence $\norma{T}=\norma{g_2}_r$. Moreover, since $0$ is not a point at which $\phi$ attains its maximum, we have $$\norma{T}^r=\norma{g_2}_r^r=\mathbb{E}[|G|^r]=\phi(0)<\phi(x_0)=\norma{S}^r.$$
Consequently, $\norma{T}<\norma{S}$.

We now consider the complex case. Let $T_{\mathbb C}$, $K_{\mathbb C}$, and $S_{\mathbb C}$ denote the canonical complexifications of $T$, $K$, and $S$, respectively. Then $S_{\mathbb C}=T_{\mathbb C}+K_{\mathbb C}$ and $K_{\mathbb C}$ is compact. It follows from $\norma{T} < \norma{T + K}$ and from
\cite[Theorem 1]{maligranda} that
$$ \norma{T_{\mathbb{C}}} = \norma{T} < \norma{T+K} = \norma{T_{\mathbb{C}} + K_{\mathbb{C}}}. $$
It remains to prove that $S_{\mathbb C}$ does not attain its norm. Suppose, to the contrary, that $S_{\mathbb C}$ attains its norm at some $z=x+iy\in S_{\ell_2(\mathbb C)}$, with $x, y \in \ell_{2}(\mathbb{R})$. 
Fix $s \in [0,1]$. Applying \cite[Lemma 1]{maligranda} with $a = (Sx)(s)$ and $b = (Sy)(s)$, we have
$$  \int_0^1 |(Sx)(s) \cos{(2\pi t)} + (Sy)(s) \sin{2\pi t}|^r \, dt  = d_r^r\, (|(Sx)(s)|^2 + |(Sy)(s)|^2)^{r/2},$$
where $d_r = \displaystyle \left (\int_0^1 |\cos{(2\pi t)}|^r \, dt \right )^{1/r}$.
Thus, using that $\norma{S} = \norma{S_{\mathbb{C}}} = \norma{S_{\mathbb{C}}(z)}$, the above identity, and Tonelli's theorem, we have
\begin{align*}
 \norma{S}^r & = \norma{S_{\mathbb C}(x+iy)}_r^r 
%%&= \int_0^1 |Sx(s) + i Sy(s)|^r \, ds  \\
 =  \int_0^1  (|(Sx)(s)|^2 + |(Sy)(s)|^2)^{r/2} \, ds  \\
& = \frac{1}{d_r^r} \int_0^1 \int_0^1 |(Sx)(s) \cos{(2\pi t)} + (Sy)(s) \sin{2\pi t}|^r \, dt \, ds \\   
& = \frac{1}{d_r^r} \int_0^1 \int_0^1 |S(x \cos{(2\pi t)} + y\sin{(2\pi t)})(s)|^r \, ds \, dt \\
& = \frac{1}{d_r^r} \int_0^1 \norma{S(x \cos{(2\pi t)} + y \sin{(2\pi t)})}_r^r \, dt \\
& \leq \frac{\norma{S}^r}{d_r^r} \int_0^1 \norma{x \cos{(2\pi t)} + y \sin{(2\pi t)}}_2^r \, dt \\
& \leq \norma{S}^r \norma{x+iy}_2^r = \norma{S}^r,
\end{align*}
where in the last line we used the inequality provided in the proof of \cite[Theorem 1]{maligranda}.
In particular, it holds that 
$$ \int_0^1 \left (\norma{S}^r\norma{x \cos{(2\pi t)} + y \sin{(2\pi t)}}_2^r - \norma{S(x \cos{(2\pi t)} + y \sin{(2\pi t)})}_r^r\right )\, dt = 0,$$
and since the integrand is a nonnegative continuous function of $t$, it vanishes for every $t\in[0,1]$, that is
$$ \norma{S}^r\norma{x \cos{(2\pi t)} + y \sin{(2\pi t)}}_2^r = \norma{S(x \cos{(2\pi t)} + y \sin{(2\pi t)})}_r^r $$
for every $t\in[0,1]$. Since $x$ and $y$ are not both zero, there exists $t_0\in[0,1]$ such that $u_{t_0} := x \cos{(2\pi t_0)} + y \sin{(2\pi t_0)} \neq 0$. Therefore,
$$ \norma{ S\left( \frac{u_{t_0}}{\norma{u_{t_0}}_{2}} \right) }_r
=\norma{S},$$
which shows that $S$ attains its norm. This contradicts the real case. Therefore, $S_{\mathbb C}$ does not attain its norm.

Therefore, $(\ell_2, L_r([0,1]))$ does not have the $CPP$.
\end{proof}

It follows from \cite[Proposition 2.4]{arongarciapelteix} that the weak maximizing property implies the compact perturbation property. Therefore, Lemma \ref{lema2} implies that the pair $(\ell_2, L_r([0,1]))$ fails the $WMP$ for every $r>2$. We can now prove our main result.

\medskip

\noindent {\it Proof of the Main Result:} If $p = q = 2$, the pair $(L_2([0,1]), L_2([0,1]))$ has the $WMP$ since $L_2([0,1])$ is a Hilbert space (see, e.g., \cite[p. 5]{arongarciapelteix}). 

If $p > 2$ or $q < 2$, then the pair $(L_p([0,1]), L_q([0,1]))$ does not have the $WMP$ by \cite[Theorem 3.2]{dantasjung}. It remains to consider the case $1 < p \leq 2 \leq q < \infty$ with $p \neq q$. First, assume that $p =2$. By the above, the pair $(\ell_2, L_q([0,1]))$ does not have the $WMP$ for every $ q > 2$, and since $\ell_2$ is isometrically isomorphic to a $1$-complemented subspace of $L_2([0,1])$ (see \cite[Theorem 6.2.13 and Theorem 6.4.2]{albiackalton} or the proof of \cite[Theorem 3.2]{dantasjung}), it follows  by  \cite[Proposition 2.2(b)]{dantasjung}
that  $(L_2([0,1]),L_q([0,1]))$ does not have the $WMP$.  Finally, we assume that $1 < p < 2$. 
Since $\ell_2$ embeds isometrically into $L_q([0,1])$ for every $q \geq 2$ (see \cite[Theorem 6.4.18]{albiackalton}), it is enough by \cite[Proposition 2.2(a)]{dantasjung} to prove that $(L_p([0,1]), \ell_2)$ does not have the $WMP$. Indeed, if $p^*$ denotes the conjugate exponent of $p$, then $p^* > 2$, and by Lemma \ref{lema2}, there exists a bounded linear operator $T: \ell_2 \to L_{p^*}([0,1])$ and a rank-one operator $K: \ell_2 \to L_{p^*}([0,1])$ such that $T+K$ does not attain its norm and $\norma{T+K} > \norma{T}$. Thus, $K^*: L_p([0,1]) \to \ell_2$ is a compact operator by Schauder's Theorem, $T^* + K^* = (T+K)^*$ does not attain its norm by \cite[Lemma 5.1]{garcia-lirola}, and 
$$ \norma{T^*+K^*} = \norma{T+K} > \norma{T} = \norma{T^*}. $$
Thus $(L_p([0,1]), \ell_2)$ fails the $CPP$, hence it fails the $WMP$. \qed

\medskip

It is worthy noticing that an alternative proof of the complex case could be obtained from the results of \cite{holtz}.

\bigskip

\noindent{\bf Acknowledgments:} The author would like to thank Lucas Roberto de Lima for recommending references \cite{Billingsley1995, rolla, talagrand}.

\noindent V. C. C. Miranda\\
Departamento de Matemática\\
Instituto de Ciências Matemáticas e de Computação \\
Universidade de São Paulo \\
13566-590 -- São Carlos - SP -- Brazil  \\
e-mail: viniciusmiranda@icmc.usp.br

\end{document}